\documentclass[12pt]{article}
\usepackage{amsmath,amssymb,amsbsy,amsfonts,amsthm,latexsym,amsopn,amstext,amsxtra,euscript,amscd,a4wide}

\newfont{\teneufm}{eufm10}
\newfont{\seveneufm}{eufm7}
\newfont{\fiveeufm}{eufm5}

\newfam\eufmfam
\textfont\eufmfam=\teneufm
\scriptfont\eufmfam=\seveneufm
\scriptscriptfont\eufmfam=\fiveeufm

  \def\bbbc{{\mathchoice
{\setbox0=\hbox{$\displaystyle\rm C$}\hbox{\hbox
to0pt{\kern0.4\wd0\vrule height0.9\ht0\hss}\box0}}
{\setbox0=\hbox{$\textstyle\rm C$}\hbox{\hbox
to0pt{\kern0.4\wd0\vrule height0.9\ht0\hss}\box0}}
{\setbox0=\hbox{$\scriptstyle\rm C$}\hbox{\hbox
to0pt{\kern0.4\wd0\vrule height0.9\ht0\hss}\box0}}
{\setbox0=\hbox{$\scriptscriptstyle\rmC$}\hbox{\hbox
to0pt{\kern0.4\wd0\vrule height0.9\ht0\hss}\box0}}}}
\def\bbbq{{\mathchoice{\setbox0=\hbox{$\displaystyle\rm Q$}\hbox{\raise
0.15\ht0\hbox
to0pt{\kern0.4\wd0\vruleheight0.8\ht0\hss}\box0}}
{\setbox0=\hbox{$\textstyle\rm Q$}\hbox{\raise
0.15\ht0\hboxto0pt{\kern0.4\wd0\vrule
height0.8\ht0\hss}\box0}}
{\setbox0=\hbox{$\scriptstyle\rm Q$}\hbox{\raise
0.15\ht0\hbox
to0pt{\kern0.4\wd0\vruleheight0.7\ht0\hss}\box0}}
{\setbox0=\hbox{$\scriptscriptstyle\rm
Q$}\hbox{\raise0.15\ht0\hbox to0pt{\kern0.4\wd0\vrule
height0.7\ht0\hss}\box0}}}} \def\bbbt{{\mathchoice
{\setbox0=\hbox{$\displaystyle\rm T$}\hbox{\hbox
to0pt{\kern0.3\wd0\vrule height0.9\ht0\hss}\box0}}
{\setbox0=\hbox{$\textstyle\rm T$}\hbox{\hbox
to0pt{\kern0.3\wd0\vrule height0.9\ht0\hss}\box0}}
{\setbox0=\hbox{$\scriptstyle\rm T$}\hbox{\hbox
to0pt{\kern0.3\wd0\vrule height0.9\ht0\hss}\box0}}
{\setbox0=\hbox{$\scriptscriptstyle\rm T$}\hbox{\hbox
to0pt{\kern0.3\wd0\vrule height0.9\ht0\hss}\box0}}}}
\def\bbbs{{\mathchoice {\setbox0=\hbox{$\displaystyle
\rm S$}\hbox{\raise0.5\ht0\hbox
to0pt{\kern0.35\wd0\vrule height0.45\ht0\hss}\hbox
to0pt{\kern0.55\wd0\vrule height0.5\ht0\hss}\box0}}
{\setbox0=\hbox{$\textstyle \rm
S$}\hbox{\raise0.5\ht0\hbox to0pt{\kern0.35\wd0\vrule
height0.45\ht0\hss}\hbox to0pt{\kern0.55\wd0\vrule
height0.5\ht0\hss}\box0}}
{\setbox0=\hbox{$\scriptstyle \rm
S$}\hbox{\raise0.5\ht0\hbox to0pt{\kern0.35\wd0\vrule
height0.45\ht0\hss}\raise0.05\ht0\hbox
to0pt{\kern0.5\wd0\vruleheight0.45\ht0\hss}\box0}}
{\setbox0=\hbox{$\scriptscriptstyle\rm
S$}\hbox{\raise0.5\ht0\hbox to0pt{\kern0.4\wd0\vrule
height0.45\ht0\hss}\raise0.05\ht0\hbox
to0pt{\kern0.55\wd0\vruleheight0.45\ht0\hss}\box0}}}}
\def\bbbz{{\mathchoice {\hbox{$\sf\textstyle
Z\kern-0.4em Z$}} {\hbox{$\sf\textstyle Z\kern-0.4em
Z$}} {\hbox{$\sf\scriptstyleZ\kern-0.3em Z$}}
{\hbox{$\sf\scriptscriptstyle Z\kern-0.2em Z$}}}}

\newtheorem{theorem}{Theorem}
\newtheorem{lemma}[theorem]{Lemma}

\def\cA{{\mathcal A}}
\def\cB{{\mathcal B}}
\def\cC{{\mathcal C}}

\def\cI{{\mathcal I}} \def\cJ{{\mathcal J}}
 \def\cL{{\mathcal L}}

\def\cS{{\mathcal S}} \def\cT{{\mathcal T}}

 \def\cZ{{\mathcal Z}}

\def\wA{\widetilde{{\mathcal A}}}
\def\hA{\widehat{{\mathcal A}}}
\def\wB{\widetilde{{\mathcal B}}}

\def\({\left(}\def\){\right)} \def\[{\left[}
\def\]{\right]} \def\<{\langle} \def\>{\rangle}

\def\fl#1{\left\lfloor#1\right\rfloor}

 \def\mand{\qquad\mbox{and}\qquad}

\newcommand{\set}[1]{\left\{#1\right\}}

\newcommand{\ba}{\backslash}

\newcommand{\E}{\mathbf{E}}

\newcommand{\F}{\mathbb{F}}
\newcommand{\Fb}{\overline{\F}}

\newcommand{\W}{\mathrm{W}}

\newcommand{\DIKT}{\mathrm{T}}
\newcommand{\DIKD}{\mathrm{D}}
\newcommand{\Eu}{\E_{\DIKT,u}}
\newcommand{\Wu}{\E_{\W,u}}

\newcommand{\bx} X
\newcommand{\by} Y
\newcommand{\bz} Z

\def\bF{\overline{\F}}

\def\JDIKT{J_\DIKT(q)}
\def\JDIKD{J_\DIKD(q)}

\def\IDIKT{I_{\mathrm{T}}(q)}
\def\IDIKD{I_{\mathrm{D}}(q)}

\newcommand{\keywords}[1]{{ \list{}{\advance\topsep by
-5ex \relax\small \leftmargin=1cm \labelwidth=0pt
\listparindent=0pt\itemindent\listparindent
\rightmargin\leftmargin}\item[\hskip\labelsep
\bfseries Keywords:] {#1} \endlist}}

\begin{document}

\pagestyle {plain}
\pagenumbering{arabic}
\title{Isomorphism classes of Doche-Icart-Kohel Curves \\ over finite fields}

\author{{\sc{Reza Rezaeian Farashahi}} \\
            {Department of Mathematical Sciences}\\
    {Isfahan University of Technology}\\
    {P.O. Box 85145Isfahan, Iran}\\
and\\
 School of Mathematics,\\
 Institute for Research in Fundamental Sciences (IPM),\\
 P.O. Box 19395-5746, Tehran, Iran\\
     {\tt farashahi@cc.iut.ac.ir}
         \vspace{1cm}\\
         {\sc{Mehran Hosseini}} \\
         {Department of Mathematical Sciences}\\
         {Isfahan University of Technology}\\
         {\tt mehran.hosseini@math.iut.ac.ir}}

\date{} \maketitle

\begin{abstract}
We give explicit formulas for the number of distinct elliptic curves over a finite field, up to
isomorphism, in two families of curves introduced by C.~Doche, T.~Icart and D.~R. Kohel.
\end{abstract}

\keywords{Elliptic curve, $j$-invariant, isomorphism, cryptography}

\paragraph*{2000 Mathematics Subject Classification:} 11G05,11T06, 14H52 


\section{Introduction}
\label{sec intro}
An elliptic curve is a smooth projective genus 1 curve, with a given rational point. Traditionally,
an elliptic curve $E$ over a filed $\F$ is represented by the  Weierstrass equation
\begin{equation}
\label{eq:Weier}
E :\quad  \by^2 + a_1\bx\by + a_3\by = \bx^3 + a_2\bx^2 + a_4\bx + a_6,
\end{equation}
where the coefficients $a_1, a_2, a_3, a_4, a_6\in \F$.  Elliptic curves can be represented by
several other models (see ~\cite[Chapter 13]{ACDFLNV} or~\cite[Chapter 2]{Wash}). In the past few
years, these alternative models have been revisited duo to cryptoraphic applications. Moreover,
some new families of elliptic curves have been proposed following the cryptographic interests. In
the cryptographic settings the curves are usually considered over finite fields $\F_q$ of $q$
element.

In this work, we first consider the family of elliptic curves introduced by C.~Doche, T.~Icart and
D.~R. Kohel~\cite{DIK} over finite fields $\F_q$ of characteristic $p \ge 5$
\begin{equation}
\label{eq:DIKT} \E_{\DIKT,u}: \quad \by^2 = \bx^3 + 3u(\bx+ 1)^2,
\end{equation}
where $u \ne 0,9/4$. Doche et. al. have build this family of elliptic curves for which the isogeny
of multiplication by three splits into 2 isogenies of degree 3. In this way, they proposed more
efficient tripling formulas leading to a fast scalar multiplication algorithm. Notice, an elliptic
curve defined over $\F_q$ with a rational $3$-torsion subgroup can be expressed in the latter form
(up to twists). They also proposed another family
\begin{equation}
\label{eq:DIKD} \E_{\DIKD,u}: \quad   \by^2 = \bx^3 + u\bx^2+16u\bx, \hspace{1.2cm}
\end{equation}
where $u \ne 0,64$ and $p \ge 3$. This family represents elliptic curves defined over $\F_q$ with a
rational $2$-torsion subgroup (up to twists). Accordingly a natural question arises about the
number of distinct (up to isomorphism) elliptic curves over $\F_q$ in above families
\eqref{eq:DIKT} and~\eqref{eq:DIKD}.

Farashahi and Shparlinski \cite{FS}, using the notion of the {\it $j$-invariant of an elliptic
curve\/}, see~\cite{ACDFLNV,Silv,Wash}, gave {\it exact} formulas for the number of distinct
elliptic curves $\E_{\DIKT,u}$ over $\F_q$ (up to isomorphism over the algebraic closure of $\F_q$)
when $u\in \F_q\ba \set{0,9/4}$. In addition, the exact formulas for the numbers of distinct
elliptic curves over a finite field  $\F_q$ in several families are provided in~\cite{Fa,FMW}. We
notice, the formulas given in \cite{FS,Fa,FMW} are exactly given in terms of $q$. Accordingly, the
open question of \cite{FS} is whether there is explicit formula for the number of distinct elliptic
curves (up to isomorphism over $\F_q$) in the family~\eqref{eq:DIKT}. We answer this question with
explicit estimates for the number of $\F_q$-isomorphism classes in the families of Doche, Icart and
Kohel. The interesting point is the numbers of $\F_q$-isomorphism classes of the
families~\eqref{eq:DIKT} and~\eqref{eq:DIKD} depends not only on $q$ but also on the number of
$\F_q$ rational points of some corresponding curve. In fact, we have interesting examples of
families of elliptic curves where the number of $\F_q$-isomorphism classes of the family is given
approximately by Hasse-Weil bound. Our counting proof techniques are slightly different in both
families~\eqref{eq:DIKT} and~\eqref{eq:DIKD}.

Throughout the paper, for a field $\F$, we denote its algebraic closure by $\overline{\F}$ and its
multiplicative subgroup by $\F^*$. The letter $p$ always denotes a prime number and the letter $q$
always denotes a prime power of $p$. As usual, $\F_q$ is a finite field of size $q$. Let $\chi_1$
denote the principal character in $\F_q$ where $p\ge 5$. Let $\chi_2$ denote the quadratic
character in $\F_q$, where $p\ge 3$. Then, for any $q$ where $p\ge 3$, we have $u=w^2$ for some $w
\in \F^*_q$ if and only if $\chi_2(u) =1$.  And, let $\chi_3$ denote a primitive cubic character in
$\F_q$ if $q \equiv 1 \pmod 3$ and $\chi_3 = \chi_1$ if $q \equiv 2 \pmod 3$. In particular, for
any field $\F_q$ of characteristic $p\ge 5$, $u=w^3$ for some $w \in \F^*_q$ if and only if
$\chi_3(u) =1$. We recall that $U = O(V)$ is equivalent to the inequality $|U|\le cV$ with some
constant $c> 0$. The cardinality of a finite set $\cS$ is denoted by~$\#\cS$.

\section{Isomorphisms of elliptic curves}
Here, we briefly recall some words on isomorphisms between elliptic curves,
see~\cite{ACDFLNV,Silv,Wash} for a general background on isomorphisms and elliptic curves. Two
elliptic curves given by Weierstrass equations~\eqref{eq:Weier} are isomorphic over a field $\F$ if
and only if there is a change of variables between them of the form:
\begin{equation*}
(x,y) \to (\alpha^2x+r,\alpha^3y+\alpha^2sx+t),
\end{equation*}
where $\alpha \neq 0$, and $\alpha,r,s,t \in \F$. We use $E_1 \cong_{\F} E_2$ to denote the
elliptic curves $E_1$ and $E_2$ are $\F$-{\it isomorphic}. If $\alpha,r,s,t \in \bF$, the two
elliptic curves are called {\it isomorphic} over $\bF$ or {\it twists} of each other.

The elliptic curve $E$ over $\F$ given by the Weierstrass equation~\eqref{eq:Weier} has the
non-zero discriminant
$$\Delta_E=-b_2^2b_8 -8b_4^3-27b_6^2+ 9b_2b_4b_6,$$
\begin{equation*}
b_2 = a_1^2+ 4a_2,\ b_4 = a_1a_3 + 2a_4,\ b_6 = a_3^2+ 4a_6,\ b_8 = a_1^2a_6 - a_1a_3a_4 + 4a_2a_6
+ a_2a_3^2- a_4^2.
\end{equation*}
And, the $j$-invariant of $E$ is explicitly defined as
$$
j(E) = (b_2^2-24b_4)^3/\Delta_E.
$$
It is known that two elliptic curves $E_1, E_2$ over a field $\F$ are isomorphic over
$\overline{\F}$ if and only if $j(E_1) = j(E_2)$, see~\cite[Proposition~III.1.4(b)]{Silv}. However,
two elliptic curves with the same $j$-invariant need not be isomorphic over $\F$.

Any elliptic curve over a finite field $\F_q$ of characteristic $p \ge 5$ can be represented by the
short Weierstrass equation
\begin{equation*} \label{eq:SWeier}
E_{a,b} :\quad \by^2 = \bx^3 + a\bx + b,
\end{equation*}
where the coefficients $a,b\in \F_q$, and $\Delta_{E_{a,b}}=-16(4a^3+27b^2)\ne 0$. The the
$j$-invariant of $E_{a,b}$ is given by
$$j(E_{a,b})=1728\frac{4a^3}{4a^3+27b^2}.$$
Two elliptic curves $E_{a_1,b_1}$, $E_{a_2,b_2}$ represented in the short Weierstrass form are
isomorphic over $\F_q$ if and only if there exists a nonzero element $\alpha \in \F_q$ with the
change of variables between them of the form
\begin{equation*}
(x,y) \to (\alpha^2x,\alpha^3y).
\end{equation*}
In other words, $E_{a_1,b_1}\cong_{\F_q} E_{a_2,b_2}$ if and only if there exists a nonzero element
$\alpha $ in $\F_q$ such that $a_1=\alpha^4a_2$ and $b_1=\alpha^6 b_2$. The short Weierstrass model
of the elliptic curve $\Eu$ given by the equation~\eqref{eq:DIKT} is represented by
\begin{equation}
\label{eq:WT} \Wu: \quad \by^2 = \bx^3+ a_u\bx+ b_u,
\end{equation}
where $a_u= -2^43^5u(u-2)$, $b_u=2^63^6u(2u^2-6u+3)$. The $j$-invariant of $\E_{\DIKT,u}$ is
obtained as
$$
j(\E_{\DIKT,u}) =\frac{6912u(u-2)^3}{4u-9}.
$$
Similarly, the $j$-invariant of $\E_{\DIKD,u}$ given by the equation~\eqref{eq:DIKD} is obtained as
$$
j(\E_{\DIKD,u}) =\frac{(u-48)^3}{u-64}.
$$

We use $\JDIKT$ and $\JDIKD$ to denote the numbers of distinct $j$-invariants of the elliptic
curves defined over $\F_q$ in the families \eqref{eq:DIKT} and~\eqref{eq:DIKD} respectively. We
have, \cite[Theorem 9]{FS}, $$\JDIKT= \fl{(3q+1)/4} \text{ if } q \equiv 1 \pmod 3 \quad \text{ and
}\quad \JDIKT=(q-1)/2 \text{ if } q \equiv 2 \pmod 3.$$ In Theorem \eqref{thm:numj DIKD} we provide
the number $\JDIKD$ for any finite field $\F_q$ of characteristic $p\ge 3$.

An $\Fb_q$-isomorphism class of elliptic curves in a given family may be an $\F_q$-isomorphism
class or divide into two or more $\F_q$-isomorphism classes. The former occurs when all elliptic
curve in the $\Fb_q$-isomorphism class are $\F_q$-isomorphic to each others, and the latter occurs
if at least two elliptic curves in the class are not $\F_q$-isomorphic, i.e., this happens if there
are non trivial quadratic twists in the $\Fb_q$-isomorphism class. Therefore, in a given family the
number of $\F_q$-isomorphism classes may not be so related to the number of $\Fb_q$-isomorphism
classes. In this paper, we use $\IDIKT$ and $\IDIKT$ to denote the numbers of $\F_q$-isomorphism
classes of the families~\eqref{eq:DIKT} and~\eqref{eq:DIKD} respectively. And, in the next sections
we give explicit formulas for these numbers.

%


\section{Tripling Doche-Icart-Kohel curves}
\label{sec: DIKT-curves}
\subsection{Preliminaries}
We consider the curves $\Eu$ given by~\eqref{eq:DIKT}
over a finite field $\F_q$ of characteristic~$p\ge 5$.
Notice, $u \notin \set{0,9/4}$, since the curve $\Eu$
is nonsingular. Then, the $j$-invariant
$j(\Eu) = F(u)$ where
\begin{equation}
\label{eq:map F-DIKT}
F(U) = \frac{2^83^3U(U-2)^3}{4U-9}.
\end{equation}

We recall that two elliptic the curves $\Eu$ and $\E_{\DIKT,v}$ are isomorphic over $\Fb_q$ if and
only if they have the same $j$-invariants. So, the number of distinct elliptic curves $\Eu$ up to
isomorphism over $\Fb_q$ equals the number of distinct values $F(u)$, for all $u\in \F_q \ba
\set{0,9/4}$. To compute this number, we consider the bivariate rational function $ F(U) -F(V) =
{g(U,V)}/{h(U,V)} $ with two relatively prime polynomials $g$ and $h$. We see that
\begin{equation*}
\begin{aligned}
g(U,V)=&4U^3V - 9U^3 + 4U^2V^2 - 33U^2V + 54U^2+ \\
&4U V^3 - 33U V^2 + 102U V - 108U - 9V^3 + 54V^2 - 108V + 72.
\end{aligned}
\end{equation*}
For a fixed value $u\in \F_q \setminus \set{0,9/4}$, let $g_u(V)=g(u,V)$ and let
\begin{equation*}
\Delta_u=-12(u(u - 2)(4u - 9)(2u^2 - 6u + 3))^2
\end{equation*}
be the discriminant of $g_u$. Furthermore, we let
$$ \cZ_u=\set{v~:~v \in \F_q\setminus\set{0,9/4}, \ v\ne u, \ g_u(v)=0}. $$
Moreover, for
$u\in \F_q \setminus\set{0,9/4},$
we let
\begin{eqnarray*}
\cJ_{\DIKT,u}=\set{v~:~v \in \F_q \ba \set{0,9/4}, \ \
\Eu \cong_{\overline{\F}_q} \E_{\DIKT,v}}.
\end{eqnarray*}
Clearly, we have
$\cJ_{\DIKT,u}=\cZ_u \cup \set{u}$.
We let
$$ \psi(U)=\frac{2(4U-9) }{ U }. $$
Recall from \cite[Lemma 6]{FS}, that for
$u\in \F_q \setminus\set{0,{9 /4}}$ with $\Delta_u\ne 0$ we obtain
\begin{equation}
\label{eq:Zu}
\cZ_u=\set{v~:~v\in \F_q \setminus\set{0,9/4},
v=-\frac{1}{3} \(u-6-wu+\frac{2(u-3)}{w}\),\; w \in
\F_q,\; w^3=\psi(u) }.
\end{equation}
For the particular case of $\Delta_u=0$, we have $\cJ_{\DIKT,u}=\set{u}$ if $u=2$ and
$\cJ_{\DIKT,u}=\set{u, -u+3}$ if $2u^2-6u+3=0$. The former case occurs for the elliptic curve $\Eu$
with $j(\Eu)=0$ and the latter case corresponds with elliptic curves $\Eu$ with $j(\Eu)=1728$.

We partition the set $\F_q\setminus \set{0,9/4}$ into
the following sets:
\begin{equation}
\label{eq:AB}
\cA_1\cup \cA_2\cup \cA_3 \cup \cB_1 \cup \cB_2,
\end{equation}
where
\begin{equation*}
\begin{aligned}
\cA_1&=\set{u \in \F_q\setminus \set{0,9/4}~:~\Delta_u \ne0, \ \chi_3( \psi(u)) \ne 1 \text{ and } q \equiv 1 \pmod 3 },\\
\cA_2&=\set{u \in\F_q\setminus \set{0,9/4}~:~\Delta_u \ne0 \text{ and } q \equiv 2 \pmod 3 },\\
\cA_3&=\set{u \in \F_q\setminus \set{0,9/4}~:~\Delta_u \ne0, \ \chi_3(\psi(u))= 1 \text{ and }q \equiv 1 \pmod 3 },\\
\cB_1&=\set{u=2},\\ \cB_2&=\set{u \in \F_q \setminus \set{0,9/4}~:~2u^2-6u+3=0}.
\end{aligned}
\end{equation*}

The cardinalities of above sets are provided in Table~\ref{Tab AB-DIKT} \cite[Table 7]{FS},
where we  let $q \equiv r \pmod {12}$.
\begin{table*}[ht]
\begin{center}
$\begin{array}{|c|c|c|c|c|c|c|c|}
\hline \quad r\quad & \;\#\cA_1 \;& \; \#\cA_2\; & \; \#\cA_3 \; & \; \#\cB_1 \; & \;\#\cB_2 \;\\
\hline
\hline \phantom{\int_{x_{x_1}}^{x^2}}\hspace{-20pt} 1 & \; \frac{2(q-1)}{3} \;& \; 0 \;& \; \frac{q-1}{3}-4\; & 1& 2 \\
\hline
\phantom{\int_{x_{x_1}}^{x^2}}\hspace{-20pt} 5 & \; 0 \;&\; q-3 \;& \; 0 \;& 1 & 0 \\
\hline
\phantom{\int_{x_{x_1}}^{x^2}}\hspace{-20pt} 7 & \; \frac{2(q-1)}{3} \;& \;0 \;& \; \frac{q-1}{3}-2 \;& 1& 0 \\
\hline
\phantom{\int_{x_{x_1}}^{x^2}}\hspace{-20pt} 11 & \;0 \;&\;q-5 \;& \; 0 \;& 1& 2 \\
\hline
\end{array}$
\end{center}
\vspace{-1mm}
\caption{Cardinalities of the sets $\cA_1, \cA_2,\cA_3, \cB_1, \cB_2$.}
\label{Tab AB-DIKT} \end{table*}

Furthermore, from \cite[Lemmas 7,8]{FS}, for all
$u\in \F_q \setminus\set{0,9/4}$ we have
\begin{equation}
\label{eq:Ju}\#\cJ_{\DIKT,u}= \left\{\begin{array}{ll}
1, & \text{ if } u \in \cA_1\cup \cB_1 ,\\
2, & \text{ if } u\in \cA_2\cup \cB_2,\\
4, & \text{ if } u\in \cA_3.
\end{array}
\right.
\end{equation}
Next, for the number $\JDIKT$ of $\Fb_q$-isomorphism class of elliptic curves in the
family~\eqref{eq:DIKT}, that is the cardinality of the image of the polynomial $F$ given
by~\eqref{eq:map F-DIKT}, we have
\begin{equation*}
\JDIKT={\# \cA_1+ \# \cB_1} + \frac{\# \cA_2 +\# \cB_2}{2}+ \frac{\#\cA_3}{4}.
\end{equation*}

Next, using Table~\ref{Tab AB-DIKT}, a precise formula for the number of
distinct curves of the family~\eqref{eq:DIKT} is obtained (see \cite[Theorem 9]{FS}).
\begin{theorem}
\label{thm:numj DIKT} For any finite field $\F_q$ of characteristic $p\ge 5$, for the number
$\JDIKT$ of distinct values of the $j$-invariant of the family~\eqref{eq:DIKT}, we have
$$ \JDIKT=
\left\{ \begin{array}{ll}
\displaystyle\fl{\frac{3q+1}{4}} & \text{ if } q \equiv 1 \pmod 3,\vspace{3pt}\\
\displaystyle \;\;{\frac{q-1}{2}} & \text{ if } q \equiv 2 \pmod 3.
\end{array}
\right.
$$
\end{theorem}

Notice, that elliptic curves $\Eu$ and
$\E_{\DIKT,v}$, where $v\in \cJ_{\DIKT,u}$, may not be
isomorphic over $\F_q$. So, for a fixed value
$u\in \F_q \setminus \set{0,9/4}$, we let
$$ \cI_{\DIKT,u}=\set{v~:~v \in \F_q \ba \set{0,9/4}, \ \Eu \cong_{\F_q} \E_{\DIKT,v}}. $$
Clearly, for all $u \in\F_q\setminus\set{0,9/4}$, we have $\cI_{\DIKT,u} \subseteq \cJ_{\DIKT,u}$.
To count the number of distinct $\F_q$-isomorphism classes of elliptic curves $\Eu$, for all $u\in
\F_q \setminus \set{0,9/4}$, we need to find the number of distinct sets $\cI_{\DIKT,u}$.
Therefore, first we shall obtain the cardinalities of the sets $\cI_{\DIKT,u}$, for all $u\in\F_q
\setminus \set{0,9/4}$.

We recall that the short Weierstrass model of the elliptic curve $\Eu$ is given by~\eqref{eq:WT} as
\begin{equation*}
\Wu: \quad \by^2 = \bx^3+ a_u\bx+ b_u,
\end{equation*}
where $a_u= -2^43^5u(u-2)$, $b_u=2^63^6u(2u^2-6u+3)$.

Let $u\in \F_q \ba \set{0,9/4}$ and let $v\in \cI_{\DIKT,u}$, where $v\ne u$. Therefore, $v\in
\cJ_{\DIKT,u}$. So, from $\eqref{eq:Zu}$ we have, $v=-\(u-6-wu+2(u-3)/w\)/3$, where
$w^3=\psi(u)=2(4u-9)/u$. Since the elliptic curves $\Eu$ and $\E_{\DIKT,v}$ are isomorphic over
$\F_q$, their Weierstrass models $\Wu$ and $\E_{\W,v}$ are $\F_q$-isomorphic. We have
$${a_v} =\(\frac{(w+1)(w-2)}{3w}\)^2a_u, \qquad {b_v}=\(\frac{(w+1)(w-2)}{3w}\)^3{b_u}.$$
This means, if $\Wu \cong_{\F_q} \E_{\W,v}$ then
$\frac{(w+1)(w-2)}{3w}$ is a quadratic residue in $\F_q$.

Therefore, if $\Eu \cong_{\F_q} \E_{\DIKT,v}$ then $w$
is the nonzero $x$-coordinate of an $\F_q$-rational
point on the elliptic curve
$$E: Y^2=3X(X+1)(X-2).$$
On the other hand, if $w$ is the nonzero $x$-coordinate of an $\F_q$-rational point on the elliptic
curve $E$ and if we let $u,v$ be distinct elements of $\F_q \ba \set{0,9/4}$ such that
$$\psi(u)=w^3, \quad v=-\frac{1}{3}(u-6-wu+\frac{2(u-3)}{w}),$$ then $\Eu \cong_{\F_q} \E_{\DIKT,v}$.
In other words, we have the following Lemma~\ref{lem:isoD}.

\begin{lemma}
\label{lem:isoD}
For all distinct elements $u,v \in\F_q\ba \set{0,9/4}$, we have
$\Eu \cong_{\F_q} \E_{\DIKT,v}$ if and only if we have
$$u=\frac{-18}{w^3-8} \mand v=\frac{2(w+1)^3}{w(w^2+2w+4)},$$
for some $w \in \F_q$ with $\chi_2\(3w(w+1)(w-2)\)=1$.
\end{lemma}


\subsection{The number of curves in the family~\eqref{eq:DIKT}}
\label{sec:DIKT} In this section, we look into the number of $\F_q$-isomorphism classes of tripling
Doche-Icart-Kohel curves. The estimate for the number $\IDIKT$ of $\F_q$-isomorphism classes of
this family is given by
$$ \IDIKT =
\left\{ \begin{array}{ll}
\displaystyle{\; (79/96)q } + O(\sqrt{q})   & \text{ if } q \equiv 1 \pmod {3},\vspace{2pt}\\
\displaystyle{\; (3/4)q} + O(\sqrt{q}) & \text{ if } q \equiv 2 \pmod {3}.
\end{array}
\right.
$$
Furthermore, in the next theorem we provide explicit formulas for the number of $\F_q$-isomorphism
classes of Doche-Icart-Kohel curves.

\begin{theorem}
\label{thm:numi DIKT} For any finite field $\F_q$ of characteristic $p\ge 5$, for the number
$\IDIKT$ of $\F_q$-isomorphism classes of the family~\eqref{eq:DIKT}, we have
$$ \IDIKT=
\left\{ \begin{array}{ll}
\displaystyle{\;(5q-2)/6 - (N_1-25)/96} & \text{ if } q \equiv 1 \pmod {24},\vspace{2pt}\\
\displaystyle{\; (5q+1)/6 - (N_1-25)/96} & \text{ if } q \equiv 13 \pmod {24}, \vspace{2pt}\\
\displaystyle{\; 5(q-1)/6- (N_1-25)/96} & \text{ if } q \equiv 7 \pmod {12}, \vspace{2pt}\\
\displaystyle{\; {q-1}-N_2/4} & \text{ if } q \equiv 2 \pmod {3},
\end{array}
\right.
$$
where $N_1$ is the number of $\F_q$-rational affine
points of the curve $\mathcal{C}$ given by
$$ \left\{
\begin{array}{l}
Y^2=3X(X+1)(X-2), \vspace{2pt}\\
Z^2=3\zeta X(\zeta X+1)(\zeta X-2), \vspace{2pt}\\
W^2=3\zeta^2 X(\zeta^2 X+1)(\zeta^2 X-2), \vspace{2pt}\\
\end{array}
\right.
$$
where $\zeta$ is a nontrivial cubic root of the unity in $\F_q$ where $q \equiv 1 \pmod 3$, and
$N_2$ is the number of $\F_q$-rational points of the Legendre curve $$\cL: Y^2=X(X-1)(X-1/3).$$
\end{theorem}

\begin{proof}
Let $\cJ_\DIKT$ be the set of $j$-invariants of the elliptic curve $\Eu$ for all $u\in \F_q \ba
\set{0,9/4}$. For $a\in \F_q$, let $i_{\DIKT}(a)$ be the set of distinct
$\F_q$\nobreakdash-isomorphism classes of elliptic curves $\Eu$ with $j(\Eu)=a$. Clearly,
$\#i_{\DIKT}(a)=0$, if $a \not\in \cJ_\DIKT$. Therefore, we have
$$ \IDIKT=\sum_{a\in \F_q}
\#i_{\DIKT}(a)=\sum_{a\in \cJ_\DIKT} \#i_{\DIKT}(a)\enspace. $$ Notice that, for $a\in \cJ_\DIKT$,
the cardinality of the set $i_{\DIKT}(a)$ equals $1$ or $2$. Let $\cT$ be the set of all $a\in
\cJ_\DIKT$ where all elliptic curves $\Eu$ with $j$-invariant $a$ are $\F_q$-isomorphic to each
others. In other words, we have
$$\cT=\set{a~:~ a\in \cJ_\DIKT~,~ \#i_{\DIKT}(a)=1}.$$
Then,
$\#i_{\DIKT}(a)=2$ if $a\in\cJ_\DIKT \ba \cT$. Thus,
\begin{equation*}
\begin{aligned}
\IDIKT=\sum_{a\in \cJ_\DIKT} \#i_{\DIKT}(a) &= \sum_{a\in \cT}
\#i_{\DIKT}(a)+ \sum_{a\in \cJ_\DIKT \ba \cT} \#i_{\DIKT}(a)\\
&= \sum_{a\in \cT} 1+ \sum_{a\in \cJ_\DIKT \ba \cT} 2= \#\cT+2(\#\cJ_\DIKT-\#\cT)\\
&=2\JDIKT- \#\cT
\enspace.
\end{aligned}
\end{equation*}

Next, we compute the cardinality of the set $\cT$. We recall the partition of the set
$\F_q\setminus \set{0,9/4}$ into
\begin{equation*}
\cA_1\cup \cA_2\cup \cA_3 \cup \cB_1 \cup \cB_2,
\end{equation*}
given by \eqref{eq:AB}. For $i=1,2,3$, let
$$\wA_i=\set{u~:~ u \in \cA_i,~j(\Eu) \in \cT},$$
and for $i=1,2$ we let
$$\wB_i=\set{u~:~ u \in \cB_i,~j(\Eu) \in \cT}.$$
Then, from \eqref{eq:Ju}, we have
\begin{equation}
\label{eq:T} \#\cT={\# \wA_1+ \# \wB_1} + \frac{\# \wA_2 +\# \wB_2}{2}
+ \frac{\#\wA_3}{4}.
\end{equation}

If $u=2\in \cB_1$ then $j(\Eu)=0$ and $\cJ_{\DIKT,u}=\set{u}$. Thus, $\#i_{\DIKT}(0)=1$ and $\#
\wB_1=1$. If $u\in \cB_2$ i.e., $2u^2-6u+3=0$, then $j(\Eu)=1728$ and $\cJ_{\DIKT,u}=\set{u,
-u+3}$. The short Weierstrass model of $\Eu$ is given by~\eqref{eq:WT} as $\Wu:~\by^2 = \bx^3
+a_u\bx $, where $a_u=-2^43^5u(u-2)$. We see that $a_u=-a_{-u+3}$. So, the elliptic curves $\Eu$
and $\E_{\DIKT, -u+3}$ are $\F_q$-isomorphic if and only if
$-1=w^4$ for some $w \in \F^*_q$. The latter is equivalent to have $q\equiv 1 \pmod 8$. Also, we
recall from Table~\ref{Tab AB-DIKT}, that
$$
\#\cB_2= \left\{ \begin{array}{ll} 2, & \text{ if } q \equiv \pm 1 \pmod {12},\\0,
& \text{ if } q \not\equiv \pm 1 \pmod {12}.
\end{array}
\right.
$$
So, we have
$$
\#i_{\DIKT}(1728)= \left\{
\begin{array}{ll}
2, & \text{ if } q \equiv 11,13,23 \pmod {24},\\
1, & \text{ if } q \equiv 1 \pmod {24},\\
0, & \text{ if} q \not\equiv \pm 1 \pmod {12},
\end{array}
\right.
$$
and
$$
\#\wB_2= \left\{
\begin{array}{ll}
2, & \text{ if } q \equiv1 \pmod {24},\\
0, & \text{ if } q \not\equiv 1 \pmod {24}.
\end{array}
\right.
$$

If $u\in \cA_1$, by~\eqref{eq:Ju} we have
$\#\cJ_{\DIKT,u}=1$. So, $\#i_{\DIKT}(a)=1$, where
$a=j(\Eu)$. Therefore, for all $u\in \cA_1$ we have
$j(\Eu) \in \cT$. Thus
$$\wA_1=\cA_1.$$
If $u\in \cA_2$, then $u\in \F_q \ba \set{0,2,9/4}$ where
$q \equiv2 \pmod {3}$, and $2u^2-6u+3\ne 0$.
By~\eqref{eq:Ju} we have $\#\cJ_{\DIKT,u}=2$. More
precisely, we have $\cJ_{\DIKT,u}=\set{u,v}$, where
$v=-\frac{1}{3}(u-6-wu+\frac{2(u-3)}{w})$ and
$w^3=\psi(u)=2(4u-9)/u$ where $w\in \F_q$. Notice that
$w$ is uniquely determined by $u$, since
$q \equiv 2 \pmod {3}$.
Furthermore, from Lemma~\ref{lem:isoD},
$\Eu \cong_{\F_q} \E_{\DIKT,v}$ if and only if we have
$\chi_2\(3w(w+1)(w-2)\)=1$.So, $j(\Eu) \in \cT$ if and
only if $\chi_2\(3w(w+1)(w-2)\)=1$. Therefore,
\begin{equation*}
\wA_2=\set{u~:~u \in \cA_2,\; u=\frac{-18}{w^3-8},\; \chi_2\(3w(w+1)(w-2)\)=1}.
\end{equation*}
Let $L$ be
the elliptic curve over $\F_q$ given by
$L : Y^2=3X(X+1)(X-2)$.
Let $P=(x,y)$be a point of $L(\F_q)$ with $y\ne 0$. Let $u=\frac{-18}{x^3-8}$.
Then, $u \ne 0, 2, 9/4$. Furthermore, $2u^2-6u+3\ne 0$.
Since if $2u^2-6u+3=0$, then $3$ is a quadratic
residue in $\F_q$ where $q \equiv 11 \pmod {12}$.
And, we have $x^2+2x-2=0$ and
$3x(x+1)(x-2)=-9x^2$, which contradicts
$\chi_2(-1)=-1$. Therefore, $u\in \cA_2$ and $u \in \wA_2$.
In other word, for $q \equiv 2 \pmod 3$,
\begin{equation*}
\wA_2=\set{u~:~u=\frac{-18}{x^3-8},\; (x,y) \in
L(\F_q), y \ne 0}.
\end{equation*}
The latter means,
$\#\wA_2=(\#L(\F_q)-4)/2$ if $q \equiv 2 \pmod 3$.
In addition,the elliptic curve $L$ is $\F_q$-isomorphic
to the Legendre curve $\cL$ and we have
\begin{equation*}
\#\wA_2= \left\{ \begin{array}{ll}
0, & \text{ if } q \equiv 1 \pmod {3} \\
(\#\cL(\F_q)-4)/2,& \text{ if } q \equiv 2 \pmod {3}\\
\end{array}
\right.
\end{equation*}

If $u\in \cA_3$, then $u\in \F_q \ba \set{0,2,9/4}$ where $q\equiv 1 \pmod {3}$, and $2u^2-6u+3\ne
0$, $\chi_3( \psi(u))=1$. From~\eqref{eq:Ju} we have $\#\cJ_{\DIKT,u}=4$. So, we write
$\cJ_{\DIKT,u}=\set{u,v_1,v_2,v_3}$, where for $i=1,2,3$,
$v_i=-\frac{1}{3}(u-6-w_iu+\frac{2(u-3)}{w_i})$ and $w_i^3=\psi(u)=2(4u-9)/u$. Here, $w_1, w_2,
w_3$ are cubic roots of $\psi(u)$ in $\F_q$. By Lemma~\ref{lem:isoD}, $\Eu \cong_{\F_q}
\E_{\DIKT,v_i}$ if and only if we have $\chi_2\(3w_i(w_i+1)(w_i-2)\)=1$. So, $j(\Eu) \in \cT$ if
and only if $\chi_2\(3w_i(w_i+1)(w_i-2)\)=1$ for $i=1,2,3$. For simplicity, we let $w=w_1$. And, we
have $w_2=\zeta w$, $w_3=\zeta^2 w$. Therefore,
\begin{equation*}
\wA_3=\set{u~:~u\in \cA_3,\; u=\frac{-18}{w^3-8},\; \chi_2\(3\zeta^i w(\zeta^iw+1)(\zeta^i
w-2)\)=1,\; i=0,1,2}.
\end{equation*}
Let
\begin{equation*}
\hA_3= \set{ u~ : ~ u=\frac{-18}{x^3-8}, \; (x,y,z,w) \in \cC(\F_q),\, yzw \ne 0 }.
\end{equation*}
Clearly, $\hA_3 \subset \wA_3$. Now, let $P=(x,y,z,w)$ be an affine point of $\cC(\F_q)$ with
$yzw\ne 0$. So, $x \ne 0, -\zeta^i, 2\zeta^i$, for $i=0,1,2$. Let $u=\frac{-18}{x^3-8}$. Then, $u
\ne 0, 2, 9/4$. So, $u \in \wA_3$ if $2u^2-6u+3 \ne 0$. Thus,
$$\wA_3=\hA_3 \ba \cB_2.$$
In addition, if $2u^2-6u+3 =0 $ then $3$
is a quadratic residue in $\F_q$ where $q \equiv 1 \pmod{12}$.
And, we have $x^2+2x-2=0$. Then,
$3x(x+1)(x-2)=9x^2\beta^2$ and for i=1,2,
$$3\zeta^i x (\zeta^i x+1)(\zeta^i x-2)= 9x^2 \zeta^{2i} \beta^{2i-1}, $$
where $\beta=(x+1)(2\zeta+1)/3$. Notice that $\beta^2=-1$. Then
\begin{equation*}
\chi_2\(3\zeta^i x(\zeta^i x+1)(\zeta^i x-2)\)=
\left\{ \begin{array}{ll} 1, & \text{ if } i=0, \\
\chi_2(\beta), & \text{ if } i =1,2.
\end{array}
\right.
\end{equation*}
We note that $\chi_2(\beta)=1$
if and only if $q\equiv 1 \pmod 8$. Therefore, we have
$2u^2-6u+3 =0$ if and only if $q \equiv 1 \pmod {24}$
and in this case $u \in \hA_3$ and $u \not\in \wA_3$.
So, we have
\begin{equation*}
\#\wA_3= \left\{
\begin{array}{ll} \#\hA_3-2, & \text{ if } q \equiv 1
\pmod {24} \\ \#\hA_3, & \text{ if } q \equiv 7,13
\pmod {24} \\0, & \text{ if } q \equiv 2 \pmod {3} \\
\end{array}
\right.
\end{equation*}
We consider the map $ \tau : \cC(\F_q) \rightarrow \F_q$ by $\tau(P)=x^3$ that is a $24:1$map for
affine points $P=(x,y,z,w)$ with $yzw\ne 0$. And, there are 25 points $P=(x,y,z,w)$ on $\cC(\F_q)$
with $yzw=0$. So, $\#\hA_3=(\#\cC(\F_q)-25)/24$.

In summary, the cardinalities of sets $\wA_i$, for $i=1,2,3$, and $\wB_i$ , for $i=1,2$, are
provided in Table~\ref{Tab ABt-DIKT}.
\begin{table*}[ht]
\begin{center}
$\begin{array}{|c|c|c|c|c|c|c|c|}
\hline \phantom{\int_{x_{x_1}}^{x^2}} & \;\#\wA_1 \;& \; \#\wA_2\; & \;\#\wA_3 \; & \; \#\wB_1 \; & \; \#\wB_2 \;\\
\hline \hline
\phantom{\int_{x_{x_1}}^{x^2}}\hspace{-20pt} q \equiv 1 \pmod {24}\; & \;\frac{2(q-1)}{3} \;& \; 0 \;& \;
\frac{\#\cC(\F_q)-25}{24}-2 \; & 1& 2 \\
\hline
\phantom{\int_{x_{x_1}}^{x^2}}\hspace{-20pt} q \equiv 13 \pmod {24} & \; \frac{2(q-1)}{3} \;& \; 0 \;& \;
\frac{\#\cC(\F_q)-25}{24}\; & 1& 0 \\
\hline
\phantom{\int_{x_{x_1}}^{x^2}}\hspace{-20pt} q \equiv 7 \pmod {12} \;\;& \; \frac{2(q-1)}{3} \;& \;0 \;& \;
\frac{\#\cC(\F_q)-25}{24} \;& 1& 0 \\
\hline
\phantom{\int_{x_{x_1}}^{x^2}}\hspace{-20pt} q \equiv 2 \pmod {3}\quad & \; 0 \;& \; \frac{\#\cL(\F_q)-4}{2}
\;& \; 0 \;& 1 &0 \\
\hline
\end{array}$
\end{center}
\vspace{-1mm}
\caption{Cardinalities of the sets $\wA_1, \wA_2, \wA_3, \wB_1,\wB_2$.}
\label{Tab ABt-DIKT}
\end{table*}

Now, recall that $\IDIKT=2\JDIKT- \#\cT$. Next, using
Theorem~\ref{thm:numj DIKT} and \eqref{eq:T}, we
compute $\IDIKT$ and conclude the proof.

\end{proof}


\section{Doubling Doche-Icart-Kohel curves}

\label{sec: DIKD-curves}

\subsection{Preliminaries}
We consider the curves $\E_{\DIKD,u}$ given by~\eqref{eq:DIKD} over a finite field $\F_q$ of
characteristic~$p\ge 3$. Notice, $u \notin \set{0,64}$, since the curve $\E_{\DIKD,u}$ is
nonsingular. The $j$-invariant $j(\E_{\DIKD,u}) = F(u)$ where
\begin{equation}
\label{eq:map F-DIKD} F(U) = \frac{(U-48)^3}{U-64}.
\end{equation}
\label{sec:DIKD} Let $G$ be the bivariate function
\begin{equation*}
\begin{split}
G(U,V) &= \frac{F(U)-F(V)}{U-V}  \\
 &= \frac{U^2V+UV^2-64U^2-208UV-64V^2+9216U+9216V-331776}{(U-64)(V-64)},
\end{split}
\end{equation*}
and $g$ be the numerator of $G$ given by
\begin{equation*}
g(U,V)=U^2V+UV^2-64U^2-208UV-64V^2+9216U+9216V-331776.
\end{equation*}
For $u\in \F_q \backslash \set{0,64}$ we define the polynomial $g_u$ by $g_u(V)=g(u,V)$, and denote
it's discriminant by  $\Delta_u$ and the set of its $\F_q$ roots by $\cZ_u$. So, we have
\begin{equation}
\label{eq:gu }g_u(V)=(u-64)V^2+(u^2-208u+9216)V+(-64u^2+9216u-331776),
\end{equation}
\begin{equation*}
\Delta_u=u(u-64)(u-48)^2,
\end{equation*}
\begin{equation*}
\mathcal{Z}_u =\{v \in \mathbb{F}_q : g_u(v)=0\}.
\end{equation*}

Recall that two elliptic curves $\E_{\DIKD,u}$ and $\E_{\DIKD,v}$ are isomorphic over
$\overline{\mathbb{F}}_q$ if and only if they have the same $j$-invariant. In other words, they are
$\F_q$-isomorphic if and only if $g_u(v)=g(u,v)=0$, where $u\ne v$.

For $u \in \mathbb{F}_q \backslash \{0,64\}$, let
$$\mathcal{J}_u=\set{\E_{\DIKD,v}:\; v \in \mathbb{F}_q \backslash \{0,64\}, \; \E_{\DIKD,v} \cong_{\overline{\F}_q} \E_{\DIKD,u}},
$$
$$\mathcal{I}_u=\set{\E_{\DIKD,v}:\; v \in \mathbb{F}_q \backslash \{0,64\}, \; \E_{\DIKD,v} \cong_{\F_q} \E_{\DIKD,u}},
$$
and let $\mathcal{N}_q$ be the set of all $\mathbb{F}_q$-isomorphisms between distinct elliptic
curves in the family~\eqref{eq:DIKD} and let $n_q= \# \mathcal{N}_q$. Similarly, let
$\overline{n}_q= \# \overline{\mathcal{N}}_q$, where $\overline{\mathcal{N}}_q$ is the set of all
$\overline{\mathbb{F}}_q$-isomorphisms between distinct elliptic curves in this family.
For $i=1,2,3$, let
$$c_i=\#\set{\mathcal{J}_u : u\in\mathbb{F}_q\backslash\{0,64\}, \; \#\mathcal{J}_u=i}.$$

\vspace*{2mm}
\subsection{Number of $\overline{\F}_q$-isomorphism classes in family~\eqref{eq:DIKD}}
In this section, we compute the number of distinct doubling Doche-Icart-Kohel curves up to twists.

For $u=0,64,48$, where $\Delta_u$ equals zero, we have
$$
g_0(V)=-64(V - 72)^2, \;\; g_{64}(V)=-4096,
$$
$$
g_{48}(V)=-16(V - 48)^2, \;\; \mathcal{J}_{48}=\{\E_{\DIKD,48}\}.
$$
Also for $u=72$ (appeared in $g_0$), we have
$$
g_{72}(V)=8(V - 72)V, \;\; \mathcal{J}_{72}=\{\E_{\DIKD,72}\}.
$$
For all other values of $u$, $\#\mathcal{J}_u=3$ if and only if $g_u$ has roots in $\mathbb{F}_q$
that is equivalent to where $\Delta_u$, the discriminant of $g_u$, is a quadratic residue in
$\mathbb{F}_q$.

%
%

\begin{lemma}
\label{lem:c_3} The number $c_3$ of $\mathbb{F}_q$-isomorphism classes of cardinality three is
given by
$$
c_3= \left\{
\begin{array}{ll}
\displaystyle \;\;{\frac{q-3}{6}} &  \text{ if }  q \equiv 0 \pmod
3,\vspace{3pt}\\
\displaystyle \;\;{\frac{q-7}{6}} &  \text{ if }  q \equiv 1 \pmod
3,\vspace{3pt}\\
\displaystyle \;\;{\frac{q-5}{6}}  & \text{ if }  q \equiv 2 \pmod 3.
\end{array}
\right.
$$
\end{lemma}
\begin{proof}
In order to find the number of $\overline{\mathbb{F}}_q$-isomorphism classes of cardinality three,
we need to enumerate $u \in \F_q\backslash\{0,64\}$ so that $\Delta_u$ is a quadratic residue in
$\mathbb{F}_q$. We instead count the number of points $(u,z)$ of the curve
\begin{equation*}
u(u-64)=z^2
\end{equation*}
except for the points with $u=0,64,48,72$, i.e. $\mathcal{O}\cap\mathbb{F}^2_q$, where
$$\mathcal{O}=\{(0,0),(64,0),(72,\pm 24),(48,16\sqrt{-3})\}.$$
The above curve has $q-1$ points. However, for every $u\in\mathbb{F}_q\backslash\{0,48,64,72\}$ that
$\Delta_u$ is a quadratic residue, there are two possible values of $z$. Additionally, every
$\overline{\mathbb{F}}_q$-isomorphism class of cardinality three is counted three times. Hence
$$
c_3=\dfrac{q-1-\#(\mathcal{O}\cap\mathbb{F}^2_q)}{6}.\
$$
\end{proof}
\begin{theorem}
\label{thm:numj DIKD} For any finite field $\F_q$ of characteristic $p\ge 3$, for the number
$\JDIKD$ of distinct values of the $j$-invariant of family~\eqref{eq:DIKD}, we have
$$
\JDIKD= \left\{
\begin{array}{ll}
\displaystyle \;\;{\frac{2q-3}{3}}  & \text{ if }  q \equiv 0 \pmod 3
,\vspace{3pt}\\
\displaystyle \;\;{\frac{2q+1}{3}} &  \text{ if }  q \equiv 1 \pmod
3,\vspace{3pt}\\
\displaystyle \;\;{\frac{2q-1}{3}}  & \text{ if }  q \equiv 2 \pmod 3.
\end{array}
\right.
$$
\end{theorem}
\vspace*{1mm}
\begin{proof}
Since there are only classes of cardinality three and one, we have
\begin{equation*}
c_1= q-2-3c_3.
\end{equation*}
Replacing the values of $c_3$ according to lemma \eqref{lem:c_3} in the following equation completes the
proof.
\begin{equation*}
\JDIKD=c_1+c_3=q-2-2c_3
\end{equation*}
\end{proof}
\subsection{Number of $\F_q$-isomorphism classes in family~\eqref{eq:DIKD}}

Here, we provide the explicit formulas for the number of $\F_q$-isomorphism classes of doubling
Doche-Icart-Kohel curves. The estimate for the number $\IDIKD$ of $\F_q$-isomorphism classes of
this family is given by
$$ \IDIKD = (19/24)q  + O(\sqrt{q}).$$
In order to have $\IDIKD$ explicitly, we compute the number of those $\overline{\F}_q$-isomorphism
classes where split into two $\F_q$-isomorphism classes.

\begin{lemma}
\label{lem:alpha} Let $p \ge 3$. For every $u , v \in \mathbb{F}_q \backslash \{0,64\} $ that
$u\neq v$ and $\E_{\DIKD,u}\cong_{\overline{\mathbb{F}}_q}\E_{\DIKD,v}$, we have
$\E_{\DIKD,u}\cong_{{\mathbb{F}}_q}\E_{\DIKD,v}$ if and only if there are elements $a,b \in \mathbb{F}_q$ satisfying
\begin{equation}
\label{eq:S}
\mathcal{S}~ :~ 32a^2=\frac{b(b+32)}{b+24},
\end{equation}
with
\begin{equation}
\label{eq:a,b} \left\{
\begin{array}{ll}
\displaystyle \;{a^2v=u+3b}
,\vspace{3pt}\\
\displaystyle \;{u=\frac{-b^2}{b+16}}.
\end{array}
\right.
\end{equation}
\noindent In fact $(u,v)$ uniquely determines $(a^2,b)$ and vice versa.
\end{lemma}
\begin{proof}
Suppose $\E_{\DIKD,u}$ and $\E_{\DIKD,v}$ are $\overline{\F}_q$-isomorphic, according to
\cite{Silv} they are isomorphic over $\F_q$, if and only if there is an elliptic curve isomorphism
$\psi$ and elements $a,b \in \mathbb{F}_q$ such that

$$
\left\{
\begin{array}{ll}
\displaystyle \;\;{\psi:E_1 \rightarrow E_2},
\vspace{3pt}\\
\displaystyle \;\;{\psi(x,y)=(a^2x+b,a^3y)},
\end{array}
\right.
$$
and
\begin{equation}
\label{eq:longman} \left\{
\begin{array}{ll}
\displaystyle \;\;{a^2v=u+3b}
,\vspace{3pt}\\
\displaystyle \;\;{16a^4v=16u+2bu+3b^2}
,\vspace{3pt}\\
\displaystyle \;\;{b(16u+bu+b^2)=0}.
\end{array}
\right.
\end{equation}
Since restrictions on $u$ and $v$, and properties of $\psi$ imply that $a \neq 0$ and $b \neq 0,-16,-24,-32$, the proof is complete.
\end{proof}

Now, one can see a one to one correspondence between points of $\mathcal{S}$ and
$\mathbb{F}_q$-isomorphisms of family \eqref{eq:DIKD} except for a few points. The exceptions are
the points $(0,0),(0,-32),(\pm i,-16)$ mentioned in the proof of Lemma~\eqref{lem:alpha}, and $(\pm
i,-48),(\pm (-3 \pm i \sqrt{3})+2,8(-3 \pm i \sqrt{3}))$ that are corresponded to isomorphisms of
the classes $\mathcal{J}_{72}$ and $\mathcal{J}_{48}$ respectively.

For $p>3$, the map
\begin{equation*}
f:(a,b)\longrightarrow \left(\frac{a(b+24)}{32},\dfrac{b}{32}\right)
\end{equation*}
injectively maps $\mathcal{S}$ to
$$
\gamma:a'^2=b'(b'+1)(b'+3/4).
$$
Furthermore, all the points on $\gamma$ are images of points on $\mathcal{S}$ except for
$(a',b')=(0,-3/4)$ that is corresponded to $b=-24$. Hence we have the set $\mathcal{E}$ of
exceptional points on $\gamma$
$$
\mathcal{E}=f(\{(0,0),(0,-32),(\pm i,-16),(\pm i,-48),(\pm (-3 \pm i \sqrt{3})+2,8(-3 \pm i
\sqrt{3}))\})\cup\{(0,-3/4)\}.
$$

If $p=3$, simplifying $\mathcal{S}$ we have the curve
\begin{equation*}
\gamma:a^2=-b+1
\end{equation*}
which has $q$ points. But similar to the previous situation, points listed below are not allowed.
\begin{equation*}
\mathcal{E}=\{(0,1),(\pm 1,0),(\pm i,-1)\}
\end{equation*}
In either case
\begin{equation*}
n_q= \# \gamma(\mathbb{F}_q) - \#\(\mathcal{E}\cap\mathbb{F}_q^2\).
\end{equation*}
%
%
Therefore,
\begin{equation*}
\begin{split}
n_q= \left\{
\begin{array}{ll}
\displaystyle \;\;{\#\gamma(\mathbb{F}_q)-3-2\chi_2(-1)}  & \text{ if }    p = 3
,\vspace{3pt}\\
\displaystyle \;\;{\#\gamma(\mathbb{F}_q)-3-4\chi_2(-1)-4\chi_2(-3)}  & \text{ if }    p \neq 3.
\end{array}
\right.
\end{split}
\end{equation*}
So,
\begin{equation*}
\begin{split}
n_q= \left\{
\begin{array}{ll}
\displaystyle \;\;{q-5}  & \text{ if }    q = 3^{2k}
,\vspace{3pt}\\
\displaystyle \;\;{q-3}  & \text{ if }    q = 3^{2k+1}
,\vspace{3pt}\\
\displaystyle \;\;{\#\gamma(\mathbb{F}_q)-11} & \text{ if }    q \equiv 1 \pmod {12}
,\vspace{3pt}\\
\displaystyle \;\;{\#\gamma(\mathbb{F}_q)-7}  & \text{ if }    q \equiv 5,7 \pmod {12}
,\vspace{3pt}\\
\displaystyle \;\;{\#\gamma(\mathbb{F}_q)-3}  & \text{ if }    q \equiv 11 \pmod {12}.
\end{array}
\right.
\end{split}
\end{equation*}

In lemma~\eqref{lem:alpha}, whether $\E_{\DIKD,u}$ and $\E_{\DIKD,v}$ are $\F_q$-isomorphic or not, equations \eqref{eq:S} and \eqref{eq:a,b} have two pairs of solutions $(\pm a,b)\in \overline{\F}^2_q$ that are correspondent to the two $\overline{\F}_q$-isomorphisms from $\E_{\DIKD,u}$ to $\E_{\DIKD,v}$. 
Therefore, there are twelve $\overline{\mathbb{F}}_q$-isomorphisms
between distinct curves in a cardinality three $\overline{\mathbb{F}}_q$-isomorphism class (two isomorphisms for each two-permutation of three $\overline{\F}_q$-isomorphic curve). So
\begin{equation*}
\overline{n}_q=12c_3.
\end{equation*}

Since every two $\overline{\mathbb{F}}_q$-isomorphic curve are either $\mathbb{F}_q$-isomorphic or
nontrivial quadratic twists of each other, every $\overline{\mathbb{F}}_q$-isomorphism class of
size three either splits into one or two $\mathbb{F}_q$-isomorphism classes. Therefore, there are
either four or twelve $\mathbb{F}_q$-isomorphisms between distinct curves in every
$\overline{\mathbb{F}}_q$-isomorphism class of size three. In other words, every eight
$\overline{\mathbb{F}}_q$-isomorphisms that are not $\F_q$-isomorphisms, increase the number of
$\mathbb{F}_q$-isomorphism classes by one. Hence

\begin{equation}
\label{eq:last} \IDIKD - \JDIKD = \frac{\overline{n}_q - n_q}{8}.
\end{equation}
\begin{theorem}
\label{thm:numi DIKD} For any finite field $\F_q$ of characteristic $p\ge 3$, for the number
$\IDIKD$ of $\F_q$-isomorphism classes of the family~\eqref{eq:DIKD}, we have
$$
\IDIKD= \left\{
\begin{array}{ll}
\displaystyle{\;\; (11q+1)/12-N/8} &   \text{ if }  q \equiv 1 \pmod {12}, \vspace{2pt}\\
\displaystyle{\;\; (11q-7)/12-N/8} &   \text{ if }  q \equiv 5 \pmod {12}, \vspace{2pt}\\
\displaystyle{\;\; (11q-5)/12-N/8} &   \text{ if }  q \equiv 7 \pmod {12}, \vspace{2pt}\\
\displaystyle{\;\; (11q-13)/12-N/8} &   \text{ if }  q \equiv 11 \pmod {12},\\
\displaystyle{\;\; (19q-27)/24} &   \text{ if }  q =3^{2k},\\
\displaystyle{\;\; (19q-33)/24} &   \text{ if }  q =3^{2k+1},
\end{array}
\right.
$$
where $N$ is the number of $\F_q$-rational points on the elliptic curve
$$L: Y^2=X(X+1)(X+3/4).$$
\end{theorem}
\begin{proof}
According to \eqref{eq:last},
$$
\IDIKD=\JDIKD+\frac{\overline{n}_q - n_q}{8}.
$$
Replacing the values of $n_q$ and $\overline{n}_q$ completes the proof.
\end{proof}


\section{Conclusion}
In this paper, we give explicit formulas for the number of distinct elliptic curves over a finite
field, up to isomorphism, in two families~\eqref{eq:DIKT} and~\eqref{eq:DIKD} of curves introduced
by Doche, Icart and Kohel. For any finite field $\F_q$ of characteristic $p\ge 5$, for the number
$\JDIKT$ of $\overline{\F}_q$-isomorphism classes of the family~\eqref{eq:DIKT}, we have
\cite[Theorem 9]{FS}
$$ \JDIKT=
\left\{ \begin{array}{ll}
\displaystyle\fl{\frac{3q+1}{4}} & \text{ if } q \equiv 1 \pmod 3,\vspace{3pt}\\
\displaystyle \;\;{\frac{q-1}{2}} & \text{ if } q \equiv 2 \pmod 3.
\end{array}
\right.
$$
We recall~\cite{DIK} that for $\F_q$ of characteristic $p\ge 5$ the family~\eqref{eq:DIKT} is the
family of elliptic curves defined over $\F_q$ with a rational $3$-torsion subgroup (up to twists).
In other words the formulas provide the number of $\overline{\F}_q$-isomorphism classes of the
family of elliptic curves over $\F_q$ with a rational $3$-torsion point. In addition, an elliptic
curve with a point of order 3 is birationally equivalent to a generalized Hessian curve. The exact
number of distinct elliptic curves with a point of order 3 up to isomorphism is also provided by
\cite{Fa}. And, the the number of $\F_q$-isomorphism classes of the family of elliptic curves over
$\F_q$ with a rational $3$-torsion point is given by
$$
  \begin{cases}
    \fl{(3(q+3)/{4}}& \text{if $q \equiv 1 \pmod {3},$}\\
    \;\;{q-1} & \text{if $q \equiv 2 \pmod {3}.$}
  \end{cases}
$$
The interesting result of this paper is Theorem~\eqref{thm:numi DIKT} providing the estimate for
the number $\IDIKT$ of $\F_q$-isomorphism classes of this family by
$$ \IDIKT =
\left\{ \begin{array}{ll}
\displaystyle{\; (79/96)q } + O(\sqrt{q})   & \text{ if } q \equiv 1 \pmod {3},\vspace{2pt}\\
\displaystyle{\; (3/4)q} + O(\sqrt{q}) & \text{ if } q \equiv 2 \pmod {3}.
\end{array}
\right.
$$
Therefore, the family of Doche, Icart and Kohel does not cover all $\F_q$-isomorphism classes of
elliptic curves with a point of order 3. And, the more intrusting point is that the number $\IDIKT$
is estimated by Hasse-Weil bound.

For the family of doubling Doche, Icart and Kohel curves over any finite field $\F_q$ of
characteristic $p\ge 5$, Theorem~\eqref{thm:numj DIKD} provides the number $\JDIKD$ of
$\overline{\F}_q$-isomorphism classes of this family~\eqref{eq:DIKD} by
$$
\JDIKD=\fl{\frac{2q+1}{3}}.
$$
This family represents elliptic curves defined over $\F_q$ with a rational $2$-torsion subgroup (up
to twists). Moreover, the number of $\F_q$-isomorphism classes of the family of elliptic curves
over $\F_q$ with a rational $2$-torsion point is given by
$$
\left\{
\begin{array}{ll}
{ 2\fl{\frac{2q+1}{3}}+2}&   \text{ if } q \equiv 1 \pmod {4} ,\vspace{2pt}\\
{2\fl{\frac{2q+1}{3}} }&   \text{ if } q \not\equiv 1 \pmod {4}.\\
\end{array}
\right.
$$
And, Theorem~\eqref{thm:numi DIKT} gives the estimate for the number $\IDIKD$ of $\F_q$-isomorphism
classes of this family by
$$ \IDIKD = (19/24)q  + O(\sqrt{q}).$$
In conclusion, we studied interesting examples of families of elliptic curves with a small subgroup
where the number of $\F_q$-isomorphism classes of the family is not exactly given by the size of
the finite field and it is estimated by Hasse-Weil bound.

%
%





\end {document}